\documentclass[a4paper,12pt,reqno]{amsart}
\usepackage{}
\usepackage{amssymb}
\usepackage{mathrsfs,amsmath,amssymb,latexsym,amsfonts, amscd}

\title{Characterization of the 4-canonical birationality of algebraic 3-folds, II}
\author{Meng Chen and Qi Zhang}

\address{\rm School of Mathematical Sciences \& LMNS, Fudan University, Shanghai 200433, China}
\email{mchen@fudan.edu.cn}

\address{\rm Department of Mathematics, University of Missouri, Columbia, MO 65211, USA}
\email{qi@math.missouri.edu}

\thanks{Supported by National Natural Science Foundation of China (grants: \#11171068, \#11231003, \#11421061)}

\newcommand{\bQ}{{\mathbb Q}}
\newcommand{\bP}{{\mathbb P}}
\newcommand{\roundup}[1]{\lceil{#1}\rceil}
\newcommand{\rounddown}[1]{\lfloor{#1}\rfloor}

\newcommand\lrw{\longrightarrow}
\newcommand\rw{\rightarrow}

\newcommand\OO{{\mathcal{O}}}
\newcommand\OX{{\mathcal{O}_X}}

\newcommand\bF{{\mathbb{F}}}

\newcommand{\lsgeq}{\succcurlyeq}
\newcommand{\lsleq}{\preccurlyeq}
\newcommand{\hatE}{\hat{E}}
\newcommand{\hatF}{\hat{F}}

\newtheorem{thm}{Theorem}[section]
\newtheorem{lem}[thm]{Lemma}
\newtheorem{cor}[thm]{Corollary}
\newtheorem{prop}[thm]{Proposition}

\newtheorem{op}[thm]{Problem}
\theoremstyle{definition}
\newtheorem{defn}[thm]{Definition}

\newtheorem{exmp}[thm]{Example}

\newtheorem{propty}[thm]{Property}
\newtheorem{rem}[thm]{Remark}
\theoremstyle{remark}

\begin{document}
\begin{abstract} For nonsingular projective 3-folds $X$ (of general type) whose geometric genus  $p_g=h^0(X, K_X)$ is $\geq 5$, the birationality of the fourth canonical map $\varphi_{4,X}=\Phi_{|4K_X|}$ was characterized by D.-Q. Zhang and the first author in 2008. This paper aims at characterizing the birationality of $\varphi_{4,X}$  for those $X$ with $p_g=4$.
\end{abstract}
\maketitle

\pagestyle{myheadings}
\markboth{\hfill M. Chen and Q. Zhang\hfill}{\hfill Characterization of the 4-canonical birationality\hfill}
\numberwithin{equation}{section}
\section{\bf Introduction}
We work over any algebraically closed field $k$ of characteristic 0. Studying pluricanonical maps has been an important way of understanding the birational geometry of projective varieties.  Denote by $\varphi_{m,X}$ (or, in short, $\varphi_m$) the pluricanonical map of a given variety $X$ of dimension $n$.  A remarkable theorem of Hacon and McKernan \cite{H-M}, Takayama \cite{Ta} and Tsuji \cite{Tsuji} shows that there exists a constant $c(n)$ so that $\varphi_{m,X}$ is birational for all $m\geq c(n)$ and for any $n$-fold $X$ of general type. 
It is known that one can take $c(1)=3$, $c(2)=5$ (see Table 1 below), and $c(3)=61$ (see Table 2 below). No explicit value for $c(n)$ is known for $n \ge 4$.

Let us first recall the results of Bombieri \cite{Bom} for minimal surfaces $S$ of general type, where $p_g=p_g(S)$ and $K^2=K_S^2$:
\medskip

\centerline{\underline{Table 1}}
\medskip

\begin{center}
{
\begin{tabular}{l|l}
 \hline\hline
$p_g\geq 4$&
$\varphi_3$ and $\varphi_4$ are both birational; \\
\hline
$p_g=3$ & $\varphi_4$ is birational;\\
&$\varphi_3$ is birational if and only if $K^2\neq 2$;\\
\hline
$p_g=2$& $\varphi_4$ is birational if and only if $K^2\neq 1$;\\
\hline
$p_g=0,1$ & $\varphi_3$ and $\varphi_4$ are both birational;\\
\hline
Any S & $\varphi_m$ is birational for all $m\geq 5$.\\
\hline\hline
\end{tabular}}\end{center}
\medskip

The known behavior of $\phi_m$ for minimal 3-folds $X$ of
general type can be summarized by the following table.

\centerline{\underline{Table 2}}

\begin{center}
{
\begin{tabular}{l|l}
 \hline\hline
$p_g\geq 4$&
$\varphi_5$ is birational (\cite[Theorem 1.2(2)]{IJM}); $\exists$ examples s.t. \\
&$\varphi_4$ is not birational (\cite[Example 1.4]{IJM} and Example \ref{ex}).\\
\hline
$p_g=3$& $\varphi_6$ is birational (\cite[Theorem 1.2(1)]{IJM}); $\exists$  examples s.t.\\
& $\varphi_5$ is not birational (\cite{C-G} or
\cite[p.\ 151, No.7]{Iano}). \\
\hline
$p_g=2$ & $\varphi_8$ is birational (\cite[Section 4]{IJM});\\
& $\exists$ examples s.t. $\varphi_7$ is not birational (\cite[p.\ 151, No.12]{Iano}).\\
\hline
$p_g=1$ & $\exists$ examples s.t. $\varphi_{13}$ is not birational (\cite[p.\ 151, No.19]{Iano}).\\
\hline
$p_g=0$ & $\exists$ examples s.t. $\varphi_{26}$ is not birational (\cite[p.\ 151, No.23]{Iano}).\\
\hline
Any X & $\varphi_m$ is birational for $m\geq 61$ (\cite{EXP3}).\\
\hline\hline
\end{tabular}}\end{center}
\medskip

A natural question arising from Table 2 is whether it is possible to characterize the birationality of $\varphi_m$ for small $m$. In 2008, D.-Q. Zhang and the first author proved the following theorem in that direction.
\medskip

\noindent{\bf Theorem 0}. (\cite[Theorem 1.3]{MZ})
{\em Let X be a minimal projective 3-fold of general type (admitting at worst canonical singularities) with geometric genus $p_g(X)\geq 5$. Then:
\begin{itemize}
\item[(1)]	$\varphi_4$ is {\bf not} birational if and only if X is birationally fibered by a family $\mathfrak{C}$ of irreducible
curves of geometric genus 2 with $(K_X\cdot C_0)=1$  for a general member $C_0\in \mathfrak{C}$.
\item[(2)] In (1) the family $\mathfrak{C}$ is birationally uniquely determined by the given 3-fold X.
\end{itemize}}


The main purpose of this paper is to characterize the birationality of $\varphi_4$ in the ``next'' case $p_g=4$.

\begin{thm}\label{m1} Let $X$ be a minimal projective 3-fold of general type with $p_g(X)=4$. Then $\varphi_4$ is {\bf not} birational if and only if $X$ has one of the following structures, where some terms of the statement are defined in ``Convention'' following Property 1.2:
\begin{itemize}
\item[(1)] $K_X^3=2$ and the canonical map $\varphi_1$ is a generically double cover onto $\bP^3$.
\item[(2)] $X$ has a genus-$2$ curve family ${\mathfrak C}$ of canonical degree $1$, i.e.  $(K_X\cdot C_0)=1$ for a general element $C_0\in {\mathfrak C}$.
\item[(3)] $X$ is canonically fibered by genus-$2$ curve family $\mathfrak{C}$ of canonical degree $6/5$ over some cubic surface in $\bP^3$.
\item[(4)] $X$ is canonically fibered by genus-$2$ curve family $\mathfrak{C}$ of canonical degree $4/3$ over the quadric cone  $\bar{\mathbb F}_2\subset \bP^3$ and the sub-family  $\overline{\Phi_{K_X}^{-1}(l)}\rightarrow l$ of $\mathfrak{C}$  satisfies Property \ref{H-I}, where $l$ denotes a general line in $\bar{\mathbb F}_2$ passing through the vertex of $\bar{\mathbb F}_2$.
\end{itemize}
The curve families ${\mathfrak C}$ in Items (2), (3) and (4)  are birationally uniquely determined by $X$.
\end{thm}

\begin{propty}\label{H-I} The genus two curve family $\overline{\Phi_{K_X}^{-1}(l)}\rightarrow l$ is birationally equivalent to a fibration $\iota:F\rw \bP^1$, namely, there is a birational morphism $F\rightarrow \overline{\Phi_{K_X}^{-1}(l)}$ and $\iota$ factors through  $\overline{\Phi_{K_X}^{-1}(l)}$ where $F$ is a nonsingular projective surface. Let $C$ be a general fiber of $\iota$. Then $h^0(F, K_F-C)=1$ and the horizontal part of $|K_F-C|$ is irreducible and reduced.
\end{propty}

\noindent{\bf Convention}. Here are the definition of some frequently used terms in this paper.
\begin{itemize}
\item[$\diamond$] We say that {\it $F$ is $C$-horizontally (or $\iota$-horizontally) integral} if $\iota:F\rightarrow \bP^1$ satisfies Property \ref{H-I}. Sometimes we abuse this concept  to any birational model of $F$ and simply say, for example, that $\overline{\Phi_{K_X}^{-1}(l)}$ is {\it $C$-horizontally integral}.

\item[$\diamond$] Usually ${\mathbb F}_2$ denotes the Hirzebruch ruled surface with the unique $(-2)$-curve section.  We denote by $\bar{\mathbb F}_2$ the cone obtained by contracting the $(-2)$-curve section on ${\mathbb F}_2$. Denote by $l$ a general line in $\bar{\mathbb F}_2$ passing through the vertex.

\item[$\diamond$] We say that a smooth projective surface $F$ (of general type) {\it is an ``$(j_1,j_2)$ surface}'' if $K_{F_0}^2=j_1$ and $p_g(F_0)=j_2$ where $F_0$ is the minimal model of $F$.

\item[$\diamond$]  A smooth projective surface $F$ is said to {\it be canonically fibered by curves} if $|K_F|$ is composed of a pencil of curves.
\end{itemize}

A direct consequence of Theorem \ref{m1} is the following corollary.

\begin{cor}\label{cr}  Let $X$ be a minimal projective 3-fold of general type with $p_g(X)=4$. Then $\varphi_4$ is either birational or generically finite of degree 2.
\end{cor}

\begin{exmp}\label{ex} (1) The general hypersurface $X=X_{10}\subset \bP(1,1,1,1,5)$ is a smooth canonical 3-fold with $p_g=4$ and $K_X^3=2$. Clearly $\varphi_{1,X}$ is a finite morphism of degree 2 onto $\bP^3$ and $\varphi_{4,X}$ is a double cover.

(2) For any projective $\bQ$-factorial terminal 3-fold $X$, which is birationally fibered by $(1,2)$ surfaces, $X$ has a natural curve family of canonical degree 1.  Clearly $\varphi_{4,X}$ is not birational by Bombieri's theorem.
\end{exmp}

\begin{rem} (1) It is not known whether items (3) and (4) in Theorem \ref{m1} actually occur.

(2) If a smooth projective surface $F$ is fibered by curves $C$ of genus $2$ and $F$ is $C$-horizontally integral, it is easy to see that either $p_g(F)=2$ or $p_g(F)=3$ and $|K_F|$ is not composed of a pencil of curves.

(3) It is unclear to the authors whether a minimal surface $S$ satisfying $K_S^2=2$ and $p_g(S)=3$ may admit a free pencil of curves of genus $2$.

(4) Item (4) in Theorem \ref{m1} suggests that some 3-folds which are fibered by $(2,3)$-surfaces may have non-birational 4-canonical maps. Of course, it is clear that these 3-folds have non-birational 3-canonical maps by Bombieri theorem.
\end{rem}

Throughout we will use the following symbols:
\begin{itemize}
\item[$\diamond$] ``$\sim$'' denotes linear equivalence or ${\mathbb Q}$-linear equivalence;
\item[$\diamond$] ``$\equiv$'' denotes numerical equivalence;
\item[$\diamond$] ``$|M_1|\lsgeq |M_2|$'' (or, equivalently,  ``$|M_2|\lsleq |M_1|$'') means, for linear systems $|M_1|$ and $|M_2|$ on a variety,
$$|M_1|\supseteq|M_2|+\text{(fixed effective divisor)}.$$
\item[$\diamond$] ``$D\leq D'$'' means that $D'-D$ is linearly (or ${\mathbb Q}$-linearly) equivalent to an effective divisor (or effective $\bQ$-divisor)  subject to the context for two divisors (or $\bQ$-divisors) $D$ and $D'$.
\end{itemize}

\section{\bf Preliminaries}
Throughout $X$ will be a minimal projective 3-fold of general type (with at worst ${\bQ}$-factorial terminal singularities) on which $\omega_X=\OX(K_X)$ is the canonical sheaf and $K_X$ a canonical divisor.

\subsection{Set up}\label{setup}
We assume $p_g(X)=h^0(X, \omega_X)\geq 2$. So we may study the birational structure of $X$ by considering the canonical map $$\varphi_{1}: X\dashrightarrow \bP^{p_g-1}$$ which is a non-constant rational map. 

{}From the very beginning we fix an effective Weil divisor $K_1\sim K_X$. Take successive blow-ups $\pi: X'\rightarrow X$, which exists by Hironaka's big theorem, such that:
\smallskip

(i) $X'$ is nonsingular and projective;

(ii) the moving part of $|K_{X'}|$ is base point free;

(iii) the union of supports of both $\pi^*(K_1)$ and exceptional divisors of $\pi$ is simple normal crossing.
\smallskip

Denote by $\tilde{g}$ the composition $\varphi_1\circ\pi$. So $\tilde{g}:
X'\rightarrow \Sigma\subseteq{\mathbb P}^{p_g(X)-1}$ is a morphism by the above assumption.
Let $X'\overset{f}\rightarrow \Gamma\overset{s}\rightarrow \Sigma$ be
the Stein factorization of $\tilde{g}$. We get the following commutative diagram:
\medskip

\begin{picture}(50,80) \put(100,0){$X$} \put(100,60){$X'$}
\put(170,0){$\Sigma$} \put(170,60){$\Gamma$} \put(115,65){\vector(1,0){53}}
\put(106,55){\vector(0,-1){41}} \put(175,55){\vector(0,-1){43}}
\put(114,58){\vector(1,-1){49}} \multiput(112,2.6)(5,0){11}{-}
\put(162,5){\vector(1,0){4}} \put(133,70){$f$} \put(180,30){$s$}
\put(92,30){$\pi$} \put(132,-6){$\varphi_1$}\put(136,40){$\tilde{g}$}
\end{picture}
\bigskip

We may write $K_{X'}=\pi^*(K_X)+E_{\pi}\sim M_1+Z_1,$ where $|M_1|$ is the moving part of $|K_{X'}|$, $Z_1$ the fixed part and
$E_{\pi}$ an effective ${\bQ}$-divisor which is a sum of distinct exceptional divisors with positive rational coefficients.
Since $h^0(X', \OO_{X'}(M_1))=h^0(\omega_X)$, we may also write $\pi^*(K_X)\sim_{\mathbb Q}
M_1+E_1'$ where $E_1'=Z_1-E_{\pi}$ is an effective ${\mathbb Q}$-divisor. Set $d_1=\dim\overline{\varphi_1(X)}=\dim(\Gamma)$. Clearly one has $1\leq d_1\leq 3$.

If $d_1=2$, a general fiber of $f$ is a smooth
projective curve of genus $\geq 2$. We say that $X$ is {\it canonically fibred by curves}.

If $d_1=1$, a general fiber $F$ of $f$ is a smooth
projective surface of general type. We say that $X$ is {\it
canonically fibred by surfaces} with invariants $(c_1^2(F_0), p_g(F_0)),$ where $F_0$ is the minimal model of $F$ via the contraction morphism $\sigma: F\rightarrow F_0$. We may write $M\equiv p_1 F$ where $p_1=\deg f_*\OO_{X'}(M_1)\ge p_g(X)-1$. Denote $b=g(\Gamma)$.

Just to fix the convention, {\it a generic irreducible element $S$ of} $|M_1|$ means either a general member of $|M_1|$ in the case of $d_1\geq 2$ or, otherwise, a general fiber $F$ of $f$.

For any integer $m>0$, $|M_m|$ denotes the moving part of $|mK_{X'}|$. Let $S_m$ be a general member of $|M_m|$ whenever $m>1$. Set
$$p=\begin{cases}
1, &\text{if}\ d_1\geq 2;\\
p_1, & \text{if}\ d_1=1.
\end{cases}$$
We always have
$$\pi^*(K_X)\equiv pS+E_1'$$
for the effective $\bQ$-divisor $E_1'$ on $X'$.

\subsection{\bf Convention.} For any linear system $|D|$ of positive dimension on a normal projective variety, we may write $|D|=\text{Mov}|D|+(\text{fixed part})$ and consider the rational map $\Phi_{|D|}=\Phi_{\text{Mov}|D|}$.  A {\it generic irreducible element} of $|D|$ means a general member of $\text{Mov}|D|$ when $|D|$ is not composed of a pencil or, otherwise, an irreducible component in a general member of $\text{Mov}|D|$.

\subsection{\bf Technical inequalities.} We refer to Chen-Zhang \cite[Section 3]{MZ} for birationality principles (see \cite[Lemma 3.1, Lemma 3.2]{MZ}).   For the convenience of readers, we briefly recall the technical, however useful,  theorem as follows.

Pick a generic irreducible element $S$ of $|M|$. Assume that we have a base point free linear system $|G|$ on $S$. Denote by $C$ a generic irreducible element of $|G|$. Since $\pi^*(K_X)|_S$ is nef and big, Kodaira's lemma implies that there is a positive rational number $\beta$ so that 
$$\pi^*(K_X)|_S-\beta C\geq 0.$$  {}From now on, we mean $\beta$ to be the supremum of all such numbers.  Set $\xi=(\pi^*(K_X)\cdot C)$ and, given any positive integer $m$, define
$$\alpha_m=(m-1-\frac{1}{p}-\frac{1}{\beta})\xi. $$
We will frequently use the following theorem.

\begin{thm}\label{key} (Chen-Zhang \cite[Theorem 3.6]{MZ})  Keep the above setting and notation. Let $m>0$ be an integer.  Then
\begin{itemize}
\item[(1)]  $m\xi\geq \deg(K_C)+\roundup{\alpha_m}$ provided that $\alpha_m>1$;

\item[(2)] $\varphi_m$ is birational provided that $|mK_{X'}||_S$ distinguishes different generic irreducible elements of $|G|$ and that $\alpha_m>2$.

\end{itemize}
\end{thm}

Note, however, that Theorem \ref{key} (1) implies
\begin{equation}\label{i1}
\xi\geq \frac{p\beta}{p\beta+p+\beta}\cdot \deg(K_C)\end{equation}
by taking a sufficiently large $m$ so that $\alpha_m>1$.

\begin{defn} Let $|N|$ be a moving linear system on a normal projective variety $Z$. We say that the rational map $\Phi_{|N|}$ {\it distinguishes sub-varieties $W_1, W_2\subset Z$} if, set theoretically, $\overline{\Phi_{|N|}(W_1)}\nsubseteqq \overline{\Phi_{|N|}(W_2)}$ and $\overline{\Phi_{|N|}(W_2)}\nsubseteqq \overline{\Phi_{|N|}(W_1)}$. We say that $\Phi_{|N|}$ {\it separates points $P, Q\in Z$} (for $P, Q\not\in \text{Bs}|N|$), if $\Phi_{|N|}(P)\neq \Phi_{|N|}(Q)$.
\end{defn}

\subsection{\bf Other required results}

We recall the following result.

\begin{thm}\label{v} (\cite[Theorem 1.5 (2)]{MA}) Let $X$ be a minimal projective 3-fold of general type with $p_g(X)\geq 4$. Then $K_X^3\geq 2$, which is optimal.
\end{thm}

The following special form of Kawamata's extension theorem will be used in our proof.

\begin{thm}\label{extens} (cf. Kawamata \cite[Theorem A]{KawaE}) Let $V$ be a smooth algebraic variety on which $D$ is a smooth divisor such that $K_V+D$ is big. Then the natural homomorphism \begin{equation}\label{mp} H^0(V, m(K_{V}+D))\lrw H^0(D, mK_D)\end{equation} 
is surjective for any integer $m>1$.
\end{thm}


\begin{cor} \label{ie} Under the setting of \ref{setup}, if $d_1=1$ and $g(\Gamma)=0$, then
$$\pi^*(K_X)|_F\sim_{\bQ} \frac{p}{p+1}\sigma^*(K_{F_0})+Q'$$
where $\sigma: F\rw F_0$ is the birational contraction onto the minimal model $F_0$ and $Q'$ is an effective $\bQ$-divisor on $F$.
\end{cor}
\begin{proof}  In order to apply Theorem \ref{extens}, we set $V=X'$ and  $D=F$. Then, for any sufficiently large and divisible integer $m>0$, one has
$|(p+1)mK_{X'}|\lsgeq |mp(K_{X'}+F)|$ and the surjective map:
\begin{equation}\label{mpp} H^0(X', mp(K_{X'}+F))\lrw H^0(F, mpK_F).\end{equation}
Since $\text{Mov}|mpK_F|=|mp\sigma^*(K_{F_0})|$ and $m(p+1)\pi^*(K_X)\geq M_{m(p+1)}$, we clearly have the following relations:
\begin{equation*}
m(p+1)\pi^*(K_X)|_F\geq M_{m(p+1)}|_F
\geq \text{Mov}(mp(K_{X'}+F))|_F\geq mp\sigma^*(K_{F_0}).
\end{equation*}
The statement follows.
\end{proof}

\section{\bf Proof of the main theorem}

Let $X$ be a minimal projective 3-fold of general type with $p_g(X)=4$. Keep the same setting as in \ref{setup}.

\subsection{Part one. $d_1=1$}\hfill

We have an induced fibration $f:X'\rw \Gamma$ whose general fiber  is $F$. Since $p_g(X)>0$ and the map $$H^0(K_{X'}-F)\rightarrow H^0(K_{X'})$$ cannot be surjective,  one has $p_g(F)>0$.

By Chen-Zhang \cite[4.8]{MZ}, we have the following  proposition:

\begin{prop}\label{>0} Assume $b=g(\Gamma)>0$. Then $\varphi_{4,X}$ is birational if and only if $F$ is not a (1,2) surface.
\end{prop}

Now we assume $b=g(\Gamma)=0$.  By definition we have $p=3$ and $K_{X'}\geq 3F$.
By Relation (\ref{mp}), one has the surjective map
\begin{equation}\label{i2}H^0(X', 3(K_{X'}+F))\lrw H^0(F, 3K_F),\end{equation}
which means that $\varphi_{4,X}$ is birational as long as $F$ is neither a (1,2) surface nor a (2,3) surface. Besides, it is clear that $\varphi_{4,X}$ is not birational when $F$ is a (1,2) surface.

\begin{lem}\label{23} If $F$ is a (2,3) surface, then $\varphi_{4,X}$ is birational.
\end{lem}
\begin{proof}  By Bombieri's theorem (cf. \cite{Bom, BPV}),  $|3\sigma^*(K_{F_0})|$ is base point free. Thus Relation (\ref{i2}) implies
$$4\pi^*(K_X)|_F\geq {M_4}|_F\geq \text{Mov}(3(K_{X'}+F))|_F\geq 3\sigma^*(K_{F_0}),$$
which gives $\pi^*(K_X)|_F\geq \frac{3}{4}\sigma^*(K_{F_0})$.

Take $|G|=|\sigma^*(K_{F_0})|$, which is base point free (see \cite[p227]{BPV}). The generic irreducible element $C$ of $|G|$ is a smooth curve of genus 3. Relation (\ref{i2}) also implies
$|4K_{X'}||_F\lsgeq |3\sigma^*(K_{F_0})|$, which distinguishes different generic $C$.

We have $p=3$, $\beta\geq \frac{3}{4}$ and 
$$\xi=(\pi^*(K_X)\cdot C)\geq \frac{3}{4}C^2=\frac{3}{2}.$$ Since $\alpha_5\geq \frac{7}{3}\xi>3$, Theorem \ref{key} (1) implies $\xi\geq \frac{8}{5}$. Now since $$\alpha_4\geq \frac{4}{3}\xi>2,$$ Theorem \ref{key} (2) implies that $\varphi_{4,X}$ is birational.
\end{proof}

Thus we have the following conclusion in the case $d_1=1$. 

\begin{cor}\label{d1} Let $X$ be a minimal projective 3-fold of general type with $p_g(X)=4$. Keep the same notation as above. Assume $d_1=1$. Then $\varphi_{4,X}$ is not birational if and only if $F$ is a (1,2) surface. When $\varphi_{4,X}$ is not birational, $X$ has a natural genus-$2$ curve family $\mathfrak{C}$ of canonical degree 1.
\end{cor}
\begin{proof} The first part follows from Proposition \ref{>0}, Relation (\ref{i2}) and Lemma \ref{23}.

When $\varphi_4$ is not birational, we have an induced fibration $f:X'\lrw \Gamma$ whose general fiber a (1,2) surface.
We may consider the relative canonical map $\Psi: X'\dashrightarrow
\mathbb{P}({f_*\omega_{X'/\Gamma}})$ over $\Gamma$. By taking further
birational modifications we may assume that $\Psi$ is a morphism
over $\Gamma$. So we have the following commutative diagram:
\medskip

\begin{picture}(50,80) \put(100,0){$X$} \put(100,60){$X'$}
\put(170,0){$\Gamma$} \put(170,60){$\mathbb{P}(f_*\omega_{X'/\Gamma})$}
\put(112,65){\vector(1,0){53}} \put(106,55){\vector(0,-1){41}}
\put(175,55){\vector(0,-1){43}} \put(114,58){\vector(1,-1){49}}
\multiput(112,2.6)(5,0){11}{-} \put(162,5){\vector(1,0){4}}
\put(133,70){$\Psi$} 
\put(92,30){$\pi$}
\put(132,-6){$\varphi_{1}$}\put(136,40){$f$}
\end{picture}
\bigskip

Clearly the general fiber of $\Psi$ is a smooth curve of genus 2. Set $\mathfrak{C}$ to be the set of fibers of $\Psi$. As been proved in Chen-Zhang \cite[4.10]{MZ}, we know $(\pi^*(K_X)\cdot \tilde{C})=1$ for a general element $\tilde{C}\in \mathfrak{C}$. The $\pi$-image of $\mathfrak{C}$ is what we have claimed on $X$.
\end{proof}
\medskip

\subsection{\bf Part two. $d_1=2$}\hfill

We have an induced fibration $f:X'\lrw \Gamma$ onto a normal surface $\Gamma$.  Pick a general member $S\in |M_1|$. We have $p=1$ by definition. Set $|G|=|M_1|_S|$. Let $C$ be a generic irreducible element of $|G|$. Clearly $C$ is a smooth curve of genus $g(C)\geq 2$.  We may write
\begin{equation}\pi^*(K_X)|_S\equiv \beta C+E_{1,S}'\end{equation}
where $0\leq E_{1,S}'\leq E_1'|_S$ and 
\begin{equation}\label{e4.1}\beta\geq \deg(s)\deg \tilde{g}(X')\geq p_g(X)-2=2.\end{equation}

\begin{lem}\label{sep} For the general member $S\in |M_1|$, $|4K_{X'}||_S$ distinguishes different generic irreducible elements of $|G|$.
\end{lem}
\begin{proof} We have $|4K_{X'}|\lsgeq |K_{X'}+\roundup{2\pi^*(K_X)}+M_1|$. On the other hand, the Kawamata-Viehweg vanishing theorem \cite{KV, VV} implies:
\begin{equation}\label{e1}
|K_{X'}+\roundup{2\pi^*(K_X)}+M_1||_S=
|K_S+\roundup{2\pi^*(K_X)}|_S|\lsgeq |K_S+\roundup{2L}|
\end{equation}
where $L=\pi^*(K_X)|_S$ is an effective nef and big $\bQ$-divisor on $S$.

For arbitrary two different generic irreducible elements $C_1$ and $C_2$, since $$2L-C_1-C_2-\frac{2}{\beta}E_{1,S}'\equiv (2-\frac{2}{\beta})L$$ is nef and big, the Kawamata-Viehweg  vanishing theorem gives the surjective map:
\begin{eqnarray*}
&&H^0(S, K_S+\roundup{2L-\frac{2}{\beta}E_{1,S}'})\\
&\lrw& H^0(C_1, K_{C_1}+D_1)\oplus H^0(C_2, K_{C_2}+D_2)\end{eqnarray*}
where $D_i=(\roundup{2L-\frac{2}{\beta}E_{1,S}'}-C_i)|_{C_i}$ with
$$\deg(D_i)\geq (2-\frac{2}{\beta})\xi>0$$ for $i=1,2$. Since $h^0(C_i, K_{C_i}+D_i)>0$, $|K_S+2L|$ clearly distinguishes $C_1$ and $C_2$.
\end{proof}

Though the proof in this part is relatively long, Lemma \ref{sep} tells us that one only needs to verify the numerical condition of Theorem \ref{key} to prove the main theorem. In fact, the rest of this part is to check the numerical condition ``$\alpha_4>2$''.

\begin{lem}\label{g>2} If $g(C)\geq 3$,  then $\varphi_{4,X}$ is birational.
\end{lem}
\begin{proof}
Since $\beta\geq  2$, it follows from Inequality (\ref{i1}) that
$$\xi\geq \frac{\deg(K_C)}{2+\frac{1}{\beta}}\geq\frac{8}{5}.$$ Then $\alpha_4\geq(4-2-\frac{1}{2})\xi\geq \frac{12}{5}>2$. By Theorem \ref{key}(2), $\varphi_{4,X}$ is birational.
\end{proof}

\begin{lem}\label{16/5}  If $g(C)=2$, $\beta\geq 3$ and $\varphi_{4,X}$ is not birational, then either $\xi=1$ or $\xi=\frac{6}{5}$ and $\deg \tilde{g}(X')=3$.
\end{lem}
\begin{proof} When $g(C)=2$, since we already know $\beta\geq 2$, we get from Inequality (\ref{i1}) that
$\xi\geq \frac{2}{1+1+\frac{1}{2}}=\frac{4}{5}$. Then, since 
$$\alpha_4\geq (2-\frac{1}{2})\xi\geq \frac{6}{5}>1,$$ Theorem \ref{key} (1) implies $\xi\geq 1$.

Assume $\xi>1$ and $\beta\geq 3$.  {}Find an integer $l_0>5$ such that $\xi\geq \frac{l_0+1}{l_0}$. Set $m'=l_0-1$ and then we have
$$\alpha_{m'}=(l_0-1-2-\frac{1}{\beta})\xi\geq (l_0-\frac{10}{3})\frac{l_0+1}{l_0}>l_0-3>1.$$
By Theorem \ref{key} (1), one gets $\xi\geq \frac{l_0}{l_0-1}$. 
Recursively running this procedure as long as $m'\geq 5$, we 
finally get $\xi\geq \frac{6}{5}$. Clearly, if $\beta>3$, the argument 
implies $\xi>\frac{6}{5}$.  When $\xi>\frac{6}{5}$, since 
$\alpha_4=(2-\frac{1}{\beta})\xi\geq \frac{5}{3}\xi>2$, $\varphi_{4,X}$ is birational by Theorem \ref{key} (2).  In other words, if $\xi>1$, $\beta\geq 3$ and $\varphi_{4,X}$ is not birational, then $\xi=\frac{6}{5}$ and $\beta=3$.  By Inequality (\ref{e4.1}), this means $\deg \tilde{g}(X')=3$.
\end{proof}

\begin{prop}\label{g2} Assume $g(C)=2$, $\beta=2$ and $\varphi_{4,X}$ is not birational. Then either $\xi=1$ or $\xi=\frac{4}{3}$, $\deg\tilde{g}(X')=2$, $\tilde{g}(X')=\bar{\mathbb F}_2$ and the general irreducible component in $\tilde{g}^{-1}(l)$ is $C$-horizontally integral, where $l$ is the general line in $\bar{\mathbb F}_2$ passing through the vertex.
\end{prop}
\begin{proof} {}First of all, if $\xi>\frac{4}{3}$, then $\varphi_{4,X}$ is birational since $$\alpha_4=(2-\frac{1}{2})\xi>2.$$ So we may and do assume $1<\xi\leq \frac{4}{3}$ from now on.

By Inequality (\ref{e4.1}), the image surface $\Sigma=\tilde{g}(X')\subset \bP^3$ has degree 2.
Classical surface theory (cf. Reid \cite[p30, Ex.19]{Reid}) says that $\Sigma$ must be either of the following surfaces:
\begin{itemize}
\item[(I)] $\Sigma=\bP^1\times\bP^1$.
\item[(II)] $\Sigma$ is the cone $\bar{\bF}_2$ obtained by blowing-down the unique $(-2)$ curve section on Hirzebruch surface $\bF_2$.
\end{itemize}

In both cases, $\Sigma$ is normal. Modulo further birational modifications, we may and do assume that $\Gamma$ dominates the minimal resolution of singularities (if any) of $\bar{\bF}_2$ (i.e. $\Gamma$ is over $\bF_2$ in the second case). By pulling back the hyperplane section of $\Sigma$ to $\Gamma$, we have a base point free divisor $H_{\Gamma}=s^*(\OO_{\Sigma}(1))$ so that $M_1\sim f^*(H_{\Gamma})$. We now analyze the structure of $H_{\Gamma}$ in details.
\medskip

{\bf Case (I).  $\Sigma=\bP^1\times\bP^1$  is impossible.}
We consider the morphism $g=s\circ f:X'\lrw \Sigma$. Since $\OO_{\Sigma}(1)\sim L_1+L_2$ with $(L_1\cdot L_2)=1$, the pull backs of $L_1$ and $L_2$ form two fiber structures on $X'$. Set $F_1=g^*(L_1)$ and $F_2=g^*(L_2)$. Then $S\geq F_1+F_2$. We see that both $F_1$ and $F_2$ are irreducible for general $L_1$ and $L_2$ since $h^0(X', S)=4$. 
Now the vanishing theorem gives
$$|K_{X'}+\roundup{2\pi^*(K_X)}+F_1+F_2||_{F_1}\lsgeq
|K_{F_1}+\roundup{2\pi^*(K_X)|_{F_1}}+C|$$ and
$$|K_{F_1}+\roundup{2\pi^*(K_X)|_{F_1}}+C||_C=|K_C+\tilde{D}_1|$$
with $\deg(\tilde{D}_1)\geq 2\xi>2$. This simply implies the birationality of $\varphi_{4,X}$ (a contradiction) and thus $\Sigma\neq \bP^1\times \bP^1$.
\medskip

{\bf Case (II). $\Sigma=\bar{\bF}_2$ implies $\xi=\frac{4}{3}$.} Denote by $\nu:\bF_2\rw \Sigma=\bar{\bF}_2$ the blow up at the singularity of $\bar{\bF}_2$. Denote $H_2=\nu^*(\OO_{\Sigma}(1))$. Then $h^0(\bF_2, H_2)=4$. Noting that
$H_2$ is a nef and big divisor on $\bF_2$, we can write
$$H_2\sim \mu G_0+nT$$ where $G_0$ is the unique section of the ruling structure with $G_0^2=-2$, $T$ is the general fiber of the ruling of $\bF_2$, $\mu$ and $n$ are integers.
Necessarily we get $n=2$ and $\mu=1$.  Let $\theta_0:\bF_2\rw \bP^1$ be the $\bP^1$-bundle fibration and $\eta_2: \Gamma\rw \bF_2$ the birational morphism. Let $f_0:X'\lrw \bP^1$ be the composition, i.e. $f_0=\theta_0\circ\eta_2\circ f$. Let $\hat{F}$ be a general fiber of $f_0$. 

\begin{lem}\label{4/3} If $\Sigma=\bar{\bF}_2$ and $\varphi_{4,X}$ is not birational, then $\xi=\frac{4}{3}$.
\end{lem}
\begin{proof} Clearly, we see $S\sim M_1\sim 2\hat{F}+
N_0$ where $N_0=f^*\eta_2^*(G_0)$.  Observing that $|S|_S|$ is composed of a pencil of curves and $$\hat{F}\cap S\equiv(\eta_2\circ f)^*(H_2\cap T),$$ we have
$S|_{\hat{F}}\sim N_0|_{\hat{F}}\sim C$. Since $K_{\hat{F}}\geq S|_{\hat{F}}$, we see $p_g(\hat{F})\geq 2$.
Denote by $\hat{\sigma}:\hat{F}\rw \hat{F}_0$ the contraction onto the minimal model. Since $(\hat{\sigma}^*(K_{\hat{F}_0})\cdot C)=\xi>1$, we see that $\hat{F}_0$ is not a $(1,2)$ surface. We may write $\pi^*(K_X)\sim 2\hat{F}+\hat{E}_1$ for some effective $\bQ$-divisor $\hat{E}_1$ on $X'$.

Consider the pencil $|2\hat{F}|\lsleq|K_{X'}|$ and the morphism $\Phi_{|2\hat{F}|}$. Clearly $f_0$ is the induced fibration of $\Phi_{|2\hat{F}|}$. Since $K_{X'}\geq 2\hatF$, the relation (\ref{mp}) in the proof of Crollary \ref{ie} implies $\pi^*(K_X)|_{\hat{F}}\geq \frac{2}{3}\hat{\sigma}^*(K_{\hat{F}_0})$ and, for a smooth fiber $C$ of $f$ contained in a general surface $\hatF$,
$$\xi=\xi_{\hatF}=(\pi^*(K_X)|_{\hatF}\cdot C)\geq
\frac{2}{3}(\hat{\sigma}^*(K_{\hat{F}_0})\cdot C)\geq \frac{4}{3}$$
as $\hatF_0$ is not a $(1,2)$ surface.
As we mentioned at the very beginning, we get $\xi=\frac{4}{3}$. 
\end{proof}

Next, we shall analyze this very special case more explicitly as follows.

Since the Kawamata-Viehweg vanishing theorem implies $$H^1(X', K_{X'}+\roundup{3\pi^*(K_X)-\hatF-\frac{1}{2}\hat{E}_1})=0,$$ we have the surjective map
\begin{eqnarray}\label{XX}
&&H^0(X', K_{X'}+\roundup{3\pi^*(K_X)-\frac{1}{2}\hat{E}_1})\cr
&\lrw& H^0(\hatF, K_{\hatF}+\roundup{3\pi^*(K_X)-\hatF-\frac{1}{2}\hat{E}_1}|_{\hatF}).
\end{eqnarray}
Notice that
\begin{equation*}
\roundup{3\pi^*(K_X)-\hatF-\frac{1}{2}\hat{E}_1}|_{\hatF}\geq\roundup{(3\pi^*(K_X)-\hatF-\frac{1}{2}\hat{E}_1)|_{\hatF}}.
\end{equation*}
Define $Q_{52}=(3\pi^*(K_X)-\hatF-\frac{1}{2}\hat{E}_1)|_{\hatF}$.
Then the birationality of $\varphi_{4,X}$ follows from that 
of $\Phi_{|K_{\hatF}+\roundup{Q_{52}}|}$.
By definition, we have $$Q_{52}=\frac{5}{2}\hatE_1|_{\hatF}=\frac{5}{2}\pi^*(K_X)|_{\hatF}$$ since $\hatF|_{\hatF}$ is trivial. Denote by $\iota=\eta_2\circ f|_{\hatF}$. Then $\iota: \hatF\lrw T\cong \bP^1$ is a fibration whose general fibre $C$ is equivalent to  $S\cap \hatF$.  Since $\pi^*(K_X)\geq S$, we may write
$$\pi^*(K_X)|_{\hatF}=\hatE_1|_{\hatF}=C_0+\sum_{i=1}^tk_iH_i+\hatE_v$$
where $C_0\sim C$, $k_i\in \bQ^+$, $H_i$ is horizontal with respect to $\iota$ for each $i$ and  $\hatE_v$ is an $\iota$-vertical effective $\bQ$-divisor on $\hatF$. Clearly we have $\sum_i k_i(H_i\cdot C)=\frac{4}{3}$.
Since $H^1(\hatF, K_{\hatF}+\roundup{\frac{3}{2}\hatE_1|_{\hatF}})=0$ by the Kawamata-Viehweg vanishing theorem, we have the surjective map:
\begin{equation} H^0(\hatF, K_{\hatF}+\roundup{\frac{3}{2}\hatE_1|_{\hatF}}+C_0)\lrw H^0(C, K_C+D_{32})\label{F2}\end{equation}
where $D_{32}=(\roundup{\frac{3}{2}\hatE_1|_{\hatF}}+C_0-C)|_C\sim\roundup{\sum_{i=1}^t(\frac{3}{2}k_i)H_i}|_C$.

\begin{lem}\label{integral} Under the above situation, $\frac{3}{2}k_i$ is an integer for each $i=1,\ldots, t$.
\end{lem}
\begin{proof} Assume, to the contrary, that $\frac{3}{2}k_i$ is not an integer for certain $i$. Then we have $\deg(D_{32})>\frac{3}{2}\xi=2$, which means $\varphi_{4, X'}|_C$ is birational. The birationality principle implies that $\varphi_4$ is birational, a contradiction.
\end{proof}

Lemma \ref{integral} implies that one of the following situations occurs:
\begin{enumerate}
\item[(i).] $t=2$, $k_1=k_2=\frac{2}{3}$ and $(H_1\cdot C_0)=(H_2\cdot C_0)=1$, $H_1\neq H_2$;
\item[(ii).] $t=1$, $k_1=\frac{4}{3}$ and $(H_1\cdot C_0)=1$;
\item[(iii).] $t=1$, $k_1=\frac{2}{3}$ and $(H_1\cdot C_0)=(H_1\cdot C_{0,\text{red}})=2$;
\item[(iv).] $t=1$, $k_1=\frac{2}{3}$, $(H_1\cdot C_{0,\text{red}})=1$ and $(H_1\cdot C_0)=2$.
\end{enumerate}
Here we denote by $C_{0,\text{red}}$ the reduced part of $C_0$. 

\begin{lem}\label{2A_0} None of the cases (i), (ii) and (iii) is possible. 
\end{lem}
\begin{proof} On the irreducible curve $H_1$, we study the divisor 
$D_{H_1}=\roundup{\frac{3}{2}\hatE_1|_{\hatF}}|_{H_1}$.
One has
\begin{eqnarray*}
\deg(D_{H_1})&\geq& \frac{3}{2}(\hatE_1|_{\hatF}\cdot H_1)+\big((\roundup{\frac{3}{2}C_0}-\frac{3}{2}C_0)\cdot H_1\big)\\
&\geq & \big((\roundup{\frac{3}{2}C_0}-\frac{3}{2}C_0)\cdot H_1\big)
\end{eqnarray*}
since $\hatE_1|_{\hatF}\sim\pi^*(K_X)|_{\hatF}$ is nef. Clearly, for cases (i), (ii) and (iii), we have $\deg(D_{H_1})>0$. Hence 
$H^1(H_1, K_{H_1}+D_{H_1})=0$. Since $$H^1({\hatF}, K_{\hatF}+\roundup{\frac{3}{2}\hatE_1|_{\hatF}})=0$$ by vanishing theorem, we have $$H^1(\hatF, K_{\hatF}+\roundup{\frac{3}{2}\hatE_1|_{\hatF}}+H_1)=0.$$
This implies that we have the surjective map
$$H^0({\hatF}, K_{\hatF}+\roundup{\frac{3}{2}\hatE_1|_{\hatF}}+C_0+H_1)
\lrw H^0(C,  K_C+\hat{D})$$
where $\hat{D}=(\roundup{\frac{3}{2}\hatE_1|_{\hatF}}+C_0-C+H_1)|_C$ with
$$\deg(\hat{D})\geq \frac{3}{2}\xi+(H_1\cdot C)\geq 3$$
for the general curve $C\sim C_0$. So $|K_{\hatF}+\roundup{\frac{3}{2}\hatE_1|_{\hatF}}+C_0+H_1|$ gives a birational map. By Relation (\ref{XX}), we see that $\varphi_{4,X'}$ is birational (a contradiction).
\end{proof}



We have proved that only (iv) is possible. 

\begin{lem}\label{Y} In case (iv), $\hat{F}$ is C-horizontally integral.
\end{lem}
\begin{proof} In case (iv), $H_1$ is clearly the unique $\iota$-horizontal component in $\pi^*(K_X)|_{\hatF}$.
{}First we consider the case that $|K_{\hatF}|$ is composed with a pencil. Then $\iota:\hatF\rw \bP^1$ must be the induced fibration from $\Phi_{|K_{\hatF}|}$. Assume $p_g(\hatF)\geq 3$. Then we have $\sigma^*(K_{\hatF_0})\sim 2C+E_0$ for an effective divisor $E_0$. Then we get
$$\pi^*(K_X)|_{\hatF}\sim \frac{4}{3}C+\tilde{E}_1$$
for certain effective $\bQ$-divisor $\tilde{E}_1$. Since
$$Q_{52}-C-\frac{3}{4}\tilde{E}_1\equiv \frac{7}{4}\pi^*(K_X)|_{\hatF}$$
is nef and big, the vanishing theorem again implies:
$$|K_{\hatF}+\roundup{Q_{52}-\frac{3}{4}\tilde{E}_1}||_C=|K_C+\tilde{D}|$$
where $\tilde{D}=\roundup{Q_{52}-C-\frac{3}{4}\tilde{E}_1}|_C$ and
$\deg(\tilde{D})\geq \frac{7}{4}\xi>2$, which means $\varphi_{4,X'}$ is birational (a contradiction). Thus $p_g(\hatF)=2$. Clearly we have $h^0(K_{\hatF}-C)=1$.

Next, we consider the case that $|K_{\hatF}|$ is not composed of a pencil. Let us assume $p_g(\hatF)\geq 4$. Modulo further birational modifications, we may and do assume that the moving part $|\tilde{C}|$ of $|K_{\hatF}|$ is base point free. Pick a general curve $\tilde{C}$. Then $$\sigma^*(K_{\hatF_0})^2\geq \tilde{C}^2\geq 2p_g(\hatF)-4\geq 4$$ and
$$(\pi^*(K_X)|_{\hatF}\cdot \tilde{C})\geq \frac{2}{3}\sqrt{\sigma^*(K_{\hatF_0})^2\cdot \tilde{C}^2}\geq \frac{8}{3}.$$
Write $\pi^*(K_X)|_{\hatF}\sim \frac{2}{3}\tilde{C}+\hatE_{00}$ for an effective $\bQ$-divisor $\hatE_{00}$ on $F$.
Since
$$Q_{52}-\tilde{C}-\frac{3}{2}\hatE_{00}\equiv \pi^*(K_X)|_{\hatF}$$
is nef and big, the vanishing theorem again implies:
$$|K_{\hatF}+\roundup{Q_{52}-\frac{3}{2}\hatE_{00}}||_{\tilde{C}}=|K_{\tilde{C}}+\tilde{D}_0|$$
where $\tilde{D}_0=\roundup{Q_{52}-\tilde{C}-\frac{3}{2}\hatE_{00}}|_{\tilde{C}}$ and
$\deg(\tilde{D}_0)\geq \frac{8}{3}>2$, which means $\varphi_{4,X'}$ is birational (a contradiction). Thus $p_g(\hatF)=3$. Let us consider the exact sequence:
$$0\rw H^0(\hatF, K_{\hatF}-C)\rw H^0(\hatF, K_{\hatF})\overset{j}\rw H^0(C, K_C)\rw\cdots$$
Clearly, since $\dim\text{Im}(j)=2$, we have $h^0(\hatF, K_{\hatF}-C)=1$. 
In both cases, we have $C\leq \pi^*(K_X)|_{\hatF}\leq K_{\hatF}$. Since the horizontal part of $\pi^*(K_X)|_{\hatF}$ is $\frac{2}{3}H_1$ and $(H_1\cdot C)=2$, $K_{\hatF}$ has the {\it unique} irreducible and reduced horizontal part $H_1$. In a word, we have shown that $\hatF$ is $C$-horizontally integral.
\end{proof}
We have proved Proposition \ref{g2}. 
\end{proof}

\begin{prop}\label{2anti} Assume $g(C)=2$, $\beta=2$, $\xi=\frac{4}{3}$, $\deg\tilde{g}(X')=2$, $\tilde{g}(X')=\bar{\mathbb F}_2$ and the general surface $\hatF$ on $X'$ in the family induced from the ruling of $\bar{\mathbb F}_2$ is $C$-horizontally integral.
Then $\varphi_{4,X}$ is not birational.
\end{prop}
\begin{proof} Naturally we are in Case (II) in the proof of Proposition \ref{g2}. We keep the same setting as there. Pick a general fiber $\hatF$ of $\iota$. Since $\hatF$ is $C$-horizontally integral and $\pi^*(K_X)|_{\hatF}=\hatE_1|_{\hatF}\leq K_{\hatF}$, the $C$-horizontal part of $\hatE_1|_{\hatF}$ is irreducible and reduced. Thus
the horizontal part of $\pi^*(K_X)|_{\hatF}$ is exactly $\frac{2}{3}H_1$ with $(H_1\cdot C)=2$.

Since, for a general fiber $\hatF$ of $f_0$,  we have
\begin{eqnarray*}
M_4|_{\hatF}&\leq& \rounddown{4\pi^*(K_X)|_{\hatF}}\\
&=&\rounddown{\frac{8}{3}H_1+(\text{vertical divisors})}\\
&=&2H_1+(\text{vertical divisors with respect to }\iota).
\end{eqnarray*}
Thus, for a general fiber $C$ of $\iota$, $(M_4\cdot C)\leq 4$. Note that $H_1|_C$ gives a $g_2^1$ of the involution on $C$. Thus $|2H_1|_C|$ gives a finite map of degree 2.
On the other hand, Relation (\ref{F2}) implies $(M_4\cdot C)\geq 4$ and $|M_4||_C$ is base point free since $\deg(D_{32})\geq 2$. This simply implies that $\varphi_4|_C$ is finite of degree 2. Thus $\varphi_{4,X}$ is generically finite of degree 2. In particular, $\varphi_{4,X}$ is not birational.
\end{proof}

\begin{prop}\label{anti} Assume $g(C)=2$. If either $\xi=1$ or $\xi=\frac{6}{5}$ and $\deg \tilde{g}(X')=3$, then $\varphi_{4,X}$ is not birational.
\end{prop}
\begin{proof}  The proof is similar in the spirit to that of Chen-Zhang \cite[Proposition 4.6]{MZ}.

{\bf 0. Notation.} Recall that we have $K_{X'}=\pi^*(K_X)+E_{\pi}$. On $X$, we set
$Z=\pi_*(Z_1)$ and $N=\pi_*(M_1)$. Clearly $K_X\sim N+Z$. Then
there is an effective $\mathbb{Q}$-divisor $E_1$, which is
supported by some exceptional divisors, such that
$\pi^*(N)=M_1+E_1$. Therefore $E_1'=\pi^*(Z)+E_1$. For a general
member $S$ of $|M_1|$, we have
$$K_{X'}|_S=\pi^*(K_X)|_S+E_{\pi}|_S=({M_1}|_S+E_1'|_S)+E_{\pi}|_S.$$
One knows that $E_{\pi}$ is
composed of all those exceptional divisors of $\pi$. Also it is clear that
$\text{Supp}(E_1)\subseteq \text{Supp}(E_{\pi})$.

{\bf 1. Further modifications to $\pi$}.  We may
take $\pi$ to be the composition of $\pi_0$, $\pi_1$ and $\pi_2$, say
$$X'\overset{\pi_2}\longrightarrow
X_2\overset{\pi_1}\longrightarrow
X_1\overset{\pi_0}\longrightarrow X$$
 where $\pi_0$ is the
resolution of the indeterminancy of the moving part of $|K_{X}|$, $\pi_1$ is the
resolution of those isolated singularities on $X_1$ which are
away from all exceptional locus of $\pi_0$,  and finally $\pi_2$ is the
minimal further modification such that $\pi^*(K_1)$ has simple normal crossing support (recall here that $K_1\sim
K_X$ is a fixed Weil divisor as in \ref{setup}). Set
$\pi_3=\pi_0\circ \pi_1$. By abuse of notations we will have a
set of divisors for $\pi_3$ similar to that for $\pi$. For example we
may write $K_{X_2}=\pi_{3}^*(K_X)+E_{\pi_3}$ where $E_{\pi_3}$
is an effective ${\mathbb Q}$-divisor. The moving part
$|M_{\pi_3}|$ of $|K_{X_2}|$ is already base point free. Write
$\pi_3^*(N)=M_{\pi_3}+E_{1, \pi_3}$ and
$$\pi_3^*(K_X)=M_{\pi_3}+E_{1,\pi_3}'$$ where $E_{1, \pi_3}$ and
$E_{1, \pi_3}'$ are both effective ${\mathbb Q}$-divisors.
Clearly $$E_{1, \pi_3}'=\pi_3^*(Z)+E_{1, \pi_3}.$$ By the
definition of $\pi_3$, $E_{\pi_3}$ is the sum of two parts
$E_{\pi_3}'+E_{\pi_3}''$ where $E_{\pi_3}'$ consists of all those
components over the indeterminancy of $\varphi_1$ while
$E_{\pi_3}''$ is totally disjoint from $E_{\pi_3}'$. Denote by
$S_{\pi_3}$ a general member of $|M_{\pi_3}|$. Then
$|M_{\pi_3}|_{S_{\pi_3}}|$ is a free pencil of genus 2 with a
general member $C_{\pi_3}$. As we have seen
$\text{Supp}(E_{\pi_3}''|_{S_{\pi_3}})=0$ and so
$$\text{Supp}(E_{\pi_3}|_{S_{\pi_3}})=
\text{Supp}(E_{1,\pi_3}|_{S_{\pi_3}}).$$

Now we have
$$t=(\pi^*(K_X)\cdot C)=(\pi_3^*(K_X)\cdot
{\pi_2}_*(C))=(\pi_3^*(K_X)\cdot C_{\pi_3}),$$ where $t=1$ or $\frac{6}{5}$.  Since $$2 =
\deg(K_{C_{\pi_3}})= (\pi_3^* (K_X) + E_{\pi_3}) \cdot C_{\pi_3}$$
and $(\pi_3^* (K_X)\cdot C_{\pi_3})=(\pi_3^*(K_X)|_{S_{\pi_3}}\cdot C_{\pi_3})=t$, we get
$$(E_{\pi_3}|_{S_{\pi_3}}\cdot C_{\pi_3})=(E_{\pi_3}\cdot
C_{\pi_3})=2-t>0.$$ Therefore $(E_{1,\pi_3}\cdot C_{\pi_3})=(E_{1,\pi_3}|_{S_{\pi_3}}\cdot C_{\pi_3})>0$. Noting
that $$\pi_2^*(E_{1,\pi_3})\leq E_{1},$$ one has
\begin{equation}\label{positive}(E_{1}|_S\cdot C)\geq (\pi_2^*(E_{1,\pi_3})|_S\cdot
C)=(E_{1,\pi_3}|_{S_{\pi_3}}\cdot C_{\pi_3})>0 .\end{equation}

{\bf 2. Main part of the proof}. As we have known, $\varphi_{4,X}$ is birational if and only if ${\varphi_4}|_S$ is
birational for the general $S$. Now on the general surface $S$, we
have a pencil $|M_1|_S|$ and ${\varphi_4}|_S$ distinguishes different
generic irreducible elements of $|M_1|_S|$. So ${\varphi_4}|_S$ is birational if and
only if ${\varphi_4}|_C$ is bira tional. We will show that
${\varphi_4}|_C=\varphi_{|2K_C|}$, which is, however, not
birational.

\begin{lem}\label{deg2} Under the assumption of Proposition \ref{anti},  $\deg(\varphi_4)\leq 2$.
\end{lem} 
\begin{proof} 
 By the Kawamata-Viehweg vanishing theorem, we have the
surjective map:
\begin{equation}\label{s1} H^0(X', K_{X'}+\roundup{2\pi^*(K_X)}+S)\longrightarrow H^0(S,
K_S+\roundup{2\pi^*(K_X)}|_S).\end{equation}
 Since
$$2\pi^*(K_X)|_S-C-\frac{1}{\beta}E_1'|_S\equiv
(2-\frac{1}{\beta})\pi^*(K_X)|_S$$ is nef and big, the vanishing
theorem gives the surjective map:
\begin{equation}\label{s2}H^0(S,K_S+\roundup{2\pi^*(K_X)|_S-\frac{1}{\beta}E_1'|_S})\longrightarrow
H^0(C, K_C+\tilde{D}),\end{equation} where
$$\tilde{D}=\roundup{2\pi^*(K_X)|_S-\frac{1}{\beta}E_1'|_S}|_C=\roundup{(2-\frac{1}{\beta})E_1'|_S}|_C$$
and $\deg(\tilde{D})\ge (2-\frac{1}{\beta})\xi\geq 2-\frac{1}{2}>1$, noting that
$$(E_1' \cdot C) = (\pi^*(K_X) \cdot C) = \xi = t.$$ So $|K_C+\tilde{D}|$ is base
point free. Denote by $M_4'$, $N_4'$ the moving parts of
$|K_{X'}+\roundup{2\pi^*(K_X)}+S|$,
$|K_S+\roundup{2\pi^*(K_X)|_S-\frac{1}{\beta}E_1'|_S}|$ respectively.
Then one has
$$4\pi^*(K_X)|_C\geq M_4|_C\ge (M_4'|_S)|_C\ge N_4'|_C\ge K_C+\tilde{D}.$$
So $$5>4t=4\pi^*(K_X)|_S\cdot C \ge M_4\cdot C=\deg(K_C+\tilde{D})\geq 4.$$ This means
${M_4}|_C\sim K_C+\tilde{D}$ and $\deg(\tilde{D}) = 2$. On the other hand, we have shown
$|M_4||_C\lsgeq |K_C+\tilde{D}|$.  Clearly $|M_4||_C=|K_C+\tilde{D}|$.
Since $\deg(\Phi_{|K_C|}) = 2$, we have
$\deg(\varphi_4) \le 2$. So $\varphi_4$ is either
birational or a generically double cover.
\end{proof}

\begin{lem}\label{nnn} Under the assumption of Proposition \ref{anti},
 $\deg(\varphi_4)> 1$.
\end{lem} 
\begin{proof} 
  We have:
\begin{equation} \label{s3} K_C\sim (K_{X'}|_S + S|_S)|_C=(\pi^*(Z)|_S|_C+(E_1|_S)|_C+(E_{\pi}|_S)|_C.
\end{equation}
Since $2 = \deg(K_C)= (\pi^* K_X + E_{\pi})
\cdot C$ and $$(E_1'|_S\cdot C)=(\pi^*(K_X)|_S\cdot C)=t,$$ we get
$(E_{\pi}|_S\cdot C)=(E_{\pi}\cdot C)=2-t>0$. As a sub-divisor of $K_C$, $(E_{\pi}|_S)|_C$ has its support
$\text{Supp}(E_{\pi}|_S)|_C$ be one of the following situations:

{\it Case A.} a single point $P$;

{\it Case B.} two different points $P$ and $Q$ on $C$.

We consider {\it Case A} and {\it Case B} separately and note
that $$E_1'|_S=\pi^*(Z)|_S+E_1|_S$$ and \text{Supp}($E_1|_S$)
$\subset$ \text{Supp}($E_{\pi}|_S$).

Suppose we are in {\it Case A}. Then $(E_{\pi}|_S)|_C=(2-t)P$. First, if
$\text{Supp}((\pi^*(Z)|_S)|_C+(E_1|_S)|_C)$ contains a point other
than $P$ (say a point R), then
$$({\pi^*(Z)|_S})|_C+(E_1|_S)|_C+(E_{\pi}|_S)|_C=P+R$$ and $R$ is not
contained in \text{Supp}$(E_1|_C)$ since, otherwise, $R$ is in
\text{Supp}($E_{\pi}|_C$), a contradiction. Thus $R\leq
(\pi^*(Z)|_S)|_C$ as an integral part because
$$(\pi^*(Z)|_S)|_C+(E_1|_S)|_C+(E_{\pi}|_S)|_C$$ is an integral
divisor. This says $\tilde{D}=P+R\sim K_C$.  If
$$\text{Supp}((\pi^*(Z)|_S+E_1|_S)|_C$$ only contains a single point,
$(\pi^*(Z)|_S)|_C+(E_1|_S)|_C=tP$ and $K_C\sim 2P$.
In this case, we have $\tilde{D}=2P$. In a word, we always have $\varphi_4|_C=\Phi_{|2K_C|}$, which is
not birational. So $\varphi_{4,X}$ is not birational onto its image.

Suppose we are in {\it Case B}. The right hand side of (\ref{s3})
must be $P+Q$ and $K_C\sim P+Q$. We also know that
 $\tilde{D}=P+Q$. Thus
$\varphi_4|_C=\Phi_{|2K_C|}$ is not birational either.
\end{proof}
We have proved Proposition \ref{anti}. 
\end{proof}

So far, we have actually proved the following result:

\begin{thm}\label{d=2} Let $X$ be a minimal projective 3-fold of general type with $p_g(X)=4$. Keep the same notation as in \ref{setup}. Assume $d_1=2$. Then $\varphi_{4,X}$ is not birational if and only if $g(C)=2$ and one of the following holds:
\begin{itemize}
\item[i.] $(\pi^*(K_X)\cdot C)=1$;
\item[ii.] $(\pi^*(K_X)\cdot C)=\frac{6}{5}$ and $\tilde{g}(X')$ is a cubic surface in $\bP^3$.
\item[iii.] $(\pi^*(K_X)\cdot C)=\frac{4}{3}$, $\tilde{g}(X')$ is the quadric cone $\bar{\mathbb F}_2$ in $\bP^3$ and $\hatF$ is $C$-horizontally integral, where $\hatF$ on $X'$ is the general irreducible component of the $\tilde{g}^{-1}(l)$ and $l$ is the line in the ruling of $\bar{\mathbb F}_2$ passing through the vertex.
\end{itemize}
\end{thm}

\subsection{\bf Part III. $d_1=3$}\hfill


We provide a concise proof for the following theorem to make this paper as self-contained as possible, though relevant statements have been partially presented in another preprint of the first author.

\begin{thm}\label{d=3} Let $X$ be a minimal projective 3-fold of general type. Assume $p_g(X)=4$ and $\varphi_1$ is generically finite. Then $\varphi_{4,X}$ is not birational if and only if $K_X^3=2$ and $\deg(\varphi_{1,X})=2$.
\end{thm}

Keep the same setting and notation as in \ref{setup}. Pick a general member $S\in |M_1|$. Consider the linear system
$|4K_{X'}|$ and its sub-system $$|K_{X'}+\roundup{2\pi^*(K_X)}+M_1|.$$ Clearly $\varphi_4$ distinguishes different general members of $|M_1|$.  By the Kawamata-Viehweg vanishing theorem, we have the following relation:
\begin{equation}\label{e1}
|K_{X'}+\roundup{2\pi^*(K_X)}+M_1||_S=
|K_S+\roundup{2\pi^*(K_X)}|_S|\lsgeq |K_S+\roundup{2L}|
\end{equation}
where $L=\pi^*(K_X)|_S$ is an effective nef and big $\bQ$-divisor on $S$.
Set $|G|=|M_1|_S|$. Pick a generic irreducible element $C$ of $|G|$. Then, since $p_g(S)>0$, $|K_S+\roundup{2L}|$ distinguishes different general curves $C$. Thus it is sufficient to prove the birationality (or non-birationality) of $\varphi_{4}|_C$.
In fact, the Kawamata-Viehweg vanishing theorem gives
$$|K_F+\roundup{2L-{E_1'}|_F}||_C=|K_C+D_3|$$
where $D_3=\roundup{2L-{E_1'}|_F-C}|_C$ with $\deg(D_3)\geq (L\cdot C)=\xi$.

\begin{lem}\label{1-1} $K_X^3>2$ if and only if $\xi>2$.
\end{lem}
\begin{proof} Pick a general member $S\in |M_1|$. We have
$$\pi^*(K_X)|_S\sim S|_S+E_1'|_S$$ and so
$$K_X^3=(\pi^*(K_X))^3\geq (\pi^*(K_X)^2\cdot S)=\xi.$$
On $S$, since $|C|$ is not composed of a pencil of curves, $C^2\geq 2$. Thus $\xi=(\pi^*(K_X)\cdot S^2)\geq C^2\geq 2$.

On the other hand, by choosing a sufficiently large and divisible integer $n>0$ so that $|n\pi^*(K_X)|$ is base point free, one applies the Hodge Index Theorem on the general member $S_{[n]}$ to get the inequality:
$$\xi=(\pi^*(K_X)\cdot S^2)=\frac{1}{n}(\pi^*(K_X)|_{S_{[n]}}\cdot S|_{S_{[n]}})\geq \sqrt{K_X^3\cdot \xi}.$$
By Theorem \ref{v}, one has $K_X^3\geq 2$. Thus it follows that
$\xi=2$ if and only $K_X^3=2$. The lemma is proved.
\end{proof}

\begin{lem}\label{1-2} $\varphi_{4,X}$ is generically finite of degree $\leq 2$; $\varphi_{4,X}$ is birational if and only if $K_X^3>2$.
\end{lem}
\begin{proof} By definition, we have $p=1$ and $\beta=1$. Then $\alpha_4=\xi$.

Assume $K_X^3>2$. Lemma \ref{1-1} implies $\xi>2$ and Theorem \ref{key} (2) implies the birationality of $\varphi_4$.

Assume $K_X^3=2$. Note that $g:X'\lrw \bP^3$ cannot be birational. We have
\begin{equation}\label{e2}
2=K_X^3\geq S^3\geq \deg(\varphi_1)\geq 2,
\end{equation}
it follows that $\varphi_{1,X}$ is generically finite of degree $2$.  This means $\varphi_1|_C$ is a double cover onto $\bP^1$. In particular, $C$ is hyperelliptic and $S|_C$ is exactly a $g_2^1$ of $C$. Note that $C$ is a curve of genus $\geq 4$ since $K_SC+C^2\geq 6$. We have
$$|K_F+2L||_C\lsgeq |K_F+2S|_S||_C=|K_C+S|_C|$$ by the vanishing theorem. This, together with the relation (\ref{e1}), implies $|M_4||_C\lsgeq |K_C+S|_C|$, where the last one is base point free with $\deg(K_C+S|_C)\geq 8$. Since $(4\pi^*(K_X)\cdot C)=4\xi=8$, we see $|M_4||_C=|K_C+S|_C|$, which gives exactly a double cover.  Clearly, since $|M_4|$ distinguishes different curves $C$,
$\varphi_4$ is generically a double cover. We are done.
\end{proof}

Theorem \ref{d=3} automatically follows from Lemma \ref{1-2}.

\bigskip
\begin{proof}[{\bf Proof of Theorem \ref{m1}}] Assume that $\varphi_4$ is not birational.  Then $X$ has the listed $4$ structures by Corollary \ref{d1}, Theorem \ref{d=2} and Theorem \ref{d=3}.

Contrarily, if $X$ has structures (1), (3) and (4), then $\varphi_{4,X}$ is not birational by Theorem \ref{d=3}, Theorem \ref{d=2}(ii) and Theorem \ref{d=2}(iii).  Assume $X$ has structure (2).  We take the birational modification $\pi:X'\rw X$  and keep the same notation as in \ref{setup}. Then automatically $d_1\leq 2$ since, otherwise, $$(K_X\cdot C_0)=(\pi^*(K_X)\cdot \hat{C})\geq 2$$ where we assume $\pi(\hat{C})=C_0$ and $\hat{C}$ is a moving curve on $X'$.

We consider the case $d_1=2$.  Note that we have another curve family $\mathfrak{C}'$ which is induced from $\varphi_{1,X}$.  Pick a general fiber $C$ of the induced fibration $f:X'\lrw \Gamma$. Suppose $\mathfrak{C}'\neq \mathfrak{C}$. Then, for a general curve
$\hat{C}$ such that $\pi(\hat{C})=C_0\in \mathfrak{C}$, we have that $f(\hat{C})$ is a curve. Then $$(\pi^*(K_X)\cdot \hat{C})\geq (M_1\cdot \hat{C})\geq 2$$ since $g(\hat{C})\geq 2$, a contradiction. Thus $\mathfrak{C}'=\mathfrak{C}$ and, in fact, $\mathfrak{C}$ is the canonical curve family. Thus $\varphi_4$ is not birational by Theorem \ref{d=2} (i).

{}Finally let us consider the case $d_1=1$.  We have an induced fibration $f:X'\lrw \Gamma$ whose general fiber is $F$.  By \cite[Lemma 4.7]{MZ} and Corollary \ref{ie}, we have $\pi^*(K_X)|_F\geq \frac{3}{4}\sigma^*(K_{F_0})$.  Still consider the curve $\hat{C}$ on $X'$ with $\pi(\hat{C})=C_0$. If $\hat{C}$ is not vertical with respect to $f$, then $$\Phi_{|M_1|}(\hat{C})=\Gamma.$$  In particular, we have $(F\cdot \hat{C})\geq 1.$ Then $$(\pi^*(K_X)\cdot \hat{C})\geq p(F\cdot \hat{C})\geq 3,$$ a contradiction.  Therefore we see $\hat{C}\subset F$ for some smooth fiber $F$ if we choose a general curve $\hat{C}$. But then
$$1=(\pi^*(K_X)\cdot \hat{C})=(\pi^*(K_X)|_F\cdot \hat{C})\geq \frac{3}{4}(\sigma^*(K_{F_0})\cdot \hat{C})$$
implies that  $(\sigma^*(K_{F_0})\cdot \hat{C})=1$.  Since $C$ is a smooth genus 2 curve, we have $K_{F_0}^2=1$ by the Hodge Index Theorem. Besides, $|C|$ must be a rational pencil on $F$ and $K_F\geq C$. All these clearly imply that  $F$ is a $(1,2)$ surface. Therefore $\varphi_{4,X}$ is not birational by Corollary \ref{d1}.
\end{proof}

{}Finally we would like to ask the following very interesting, but challenging question:

\begin{op} (1) Is it possible to characterize the birationality of $\varphi_m$ ($m=4$, $5$) for minimal projective 3-folds $X$ of general type with $p_g=3$?

(2) Is it possible to characterize the birationality of $\varphi_m$ ($m=4$, $5$, $6$) for  minimal projective 3-folds $X$ of general type with $p_g=2$?
\end{op}

\noindent{\bf Acknowledgment}. The authors would like to thank the referee for very constructive suggestions which greatly help them improve the readability of this paper. This article was started while Chen was visiting Universit$\ddot{\text{a}}$t Bayreuth in February of 2012. Chen would like to thank Ingrid Bauer and Fabrizio Catanese for their hospitality and their generous support. Chen feels indebted to Fabrizio Catanese for many stimulating discussions. Zhang would like to thank LMNS of Fudan University for the visiting fellow support in 2012 and 2013.

\end{document}